\documentclass[12pt]{amsart}
\usepackage{graphicx}
\usepackage{amssymb}
\usepackage{stmaryrd}
\usepackage[russian,english]{babel}
\newtheorem{theorem}{Theorem}[section]
\newtheorem{lemma}[theorem]{Lemma}
\newtheorem{proposition}[theorem]{Proposition}

\newtheorem{definition}{Definition}
\newtheorem{example}[theorem]{Example}

\newtheorem{question}{Question}

\newcommand{\Snake}{\operatorname{Snake}}
\newcommand{\ind}{\operatorname{ind}}

\author{Ivan Mitrofanov}
\address{I.M.: C.N.R.S., \'{E}cole Normale Superieur, PSL Research University, France }
\email{phortim@yandex.ru}
\thanks{The author has received funding from the European Research Council (ERC) under
the European Union's Horizon 2020 
research and innovation program (grant agreement
No.725773)}

\title{{Total orders on compact metric spaces and covering dimension}}

\begin{document}

\begin{abstract}

We prove that for a compact metric space  the property of having finite covering dimension is equivalent to the existence of a total order with finite snake number.

\end{abstract}

\maketitle

\noindent\rule{12.7cm}{1.0pt}

\noindent
\textbf{
Keywords:}
total order $\cdot$
covering dimension $\cdot$
small inductive dimension.

\noindent\rule{12.7cm}{1.0pt}

\section{Introduction}\label{intro}

The following concept was introduced in \cite{EM1}.
Let $M$ be a metric space and let $T$ be a total order on $M$ (that is, for any distinct points $x, y \in M$ either $x<_Ty$ or $y<_Tx$).
Then for the ordered space $(M,T)$ one calculates {\it the order breakpoint} $\text{Br}(M,T) \in \mathbb{N} \cup \{+\infty\}$. 
The value $\min_{T}\text{Br}(M,T)$ is a quasi-isometry invariant for uniformly discrete metric spaces.

It was shown in \cite{EM2} that a metric space with finite
Assouad-Nagata dimension admits an order $T$ for which $\text{Br}(M,T)<\infty$.
It was also shown that some properties contradicting to finiteness of Assoud-Nagata dimension also forbid orders with finite $\text{Br}$,
and it was raised a question whether the existence of orders with finite $\text{Br}$ is equivalent to finiteness of Assoud-Nagata dimension.

For metric spaces there are various definitions of dimension based on coverings by sets with bounded intersection multiplicity.
In order to establish similar results for {\it Lebesgue covering dimension}, we introduce the notion of {\it snake number} as an analogue of $\text{Br}$.

\begin{definition}
	Let $M$ be a metric space, $T$ be a total order on $M$, and let $(U_1, U_2)$ be an ordered pair of subsets of $M$.
	We say that $(U_1, U_2)$ \emph{contains a snake of length $s$} if there exists a sequence of points 
	$$a_0 <_T a_1 <_T \dots <_T a_s$$ 
	such that
	$U_1$ contains the points $a_0, a_2, \dots, a_{2\lfloor s/2\rfloor}$ 
	and $U_2$ contains the points 
	$a_1, a_3, \dots, a_{2\lceil s/2 \rceil - 1}$.
\end{definition}

\begin{definition}
Let $(M,T)$ be an ordered metric space and let $x,y$ be distinct points of $M$.
\emph{The snake number } $\Snake_T(x,y)$ is the greatest number $s$ such that for any $\varepsilon > 0$ the pair of neighborhoods 
$(B_{\varepsilon}(x), B_{\varepsilon}(y))$ contains a snake of length $s$.
If the greatest such number does not exist, 
we say that $\Snake(x,y) = \infty$.
\end{definition}

\begin{definition}
Let $M$ be a metric space and $T$ be a total order on $M$. 
The \emph{Snake number} of the ordered space $(M,T)$ is defined as
$$
\Snake(M,T):=\sup_{x,y \in M, x \neq y} \Snake_T(x,y).
$$
\end{definition}

The main result of this paper is

\begin{theorem}\label{thm:main}
Let $M$ be a compact metric space.
Then the following conditions are equivalent:
\begin{enumerate}
	\item the covering dimension $\dim M < \infty$
	\item there exists an order $T$ on $M$ such that $\Snake(M,T) < \infty$.
\end{enumerate}
\end{theorem}

Let us specify the upper and lower bounds.

\begin{theorem}\label{thm:A}
	Let $(M, T)$ be an ordered compact metric space such that $\Snake(M,T) \leqslant s$.
	Then $\dim M \leqslant s$.
\end{theorem}

\begin{theorem}\label{thm:B}
	Let $M$ be a compact metric space with $\dim M \leqslant s$.
Then there exists an order $T$ for which $\Snake(M,T) \leqslant 2s + 1$.
\end{theorem}

One can find similarities with the following classical result.

\begin{theorem}[\cite{Hur}]
Let $M$ be a compact metric space with covering dimension $\dim M \geq n$.
Suppose
$f:A \to M$  is a continuous mapping of a closed subspace $A$ of the Cantor set $C$ onto the space $M$. 
Then there exist $n+1$ points $x_0,\dots, x_n \in A$ such that 
$$f(x_1) = \dots = f(x_n)$$.
\end{theorem}
(In fact, Hurewitz proved this result for small inductive dimension $\ind$ that is equal to $\dim$ for compact metric spaces). 
From Theorem \ref{thm:A} one can conclude by passing to the limit the following fact

\begin{proposition}
	Let $M$ be a compact metric space with $\dim M \geq n$.
	Suppose
	$f:A \to M$  is a continuous mapping of a closed subspace $A$ of $\mathbb{R}$ onto the space $M$. 
	Then there exist $n+1$ points $x_0 < \dots < x_n$ such that 
	$$f(x_0) = f(x_2) = \dots =f(x_{2\lfloor n/2 \rfloor})$$ and
	$$f(x_1) = f(x_3) = \dots =f(x_{2\lceil n/2 \rceil - 1}).$$
	
\end{proposition}

\section{Preliminaries}
\subsection{Snake numbers of a segment and of a circle}

\begin{example}\label{ex:1}
Obviously, for the segment $I = [0,1]$ and the natural order $T$ ($x <_T y$ if $x<y$) we have $\Snake(I,T) = 1$.
\end{example}

\begin{example}\label{ex:2}
If we glue the ends of the segment, we obtain the circle $S^1$. 
We take the order $T$ from Example \ref{ex:1} and specify that $0 = 1 <_T x$ for any other point $x$. 
It is easy to see that if $0 <x,y < 1$, then $\Snake(x ,y) = 1$, and $\Snake(0,x) = 2$, so $\Snake(S^1, T) = 2$.
\end{example}

\begin{proposition}
$\Snake(S^1, T') \geqslant 2$ for any order $T'$. 
\end{proposition}

\begin{proof}
	For any point $x\in S^1$ there is an antipodal point $A(x)$.
	We colour each point $x$ blue if $x <_T A(x)$ or red if $x >_T A(x)$.
	Since $S^1$ contains points of both colours, 
	there exists $x_0$ that belongs to the boundary of the blue set.
	For a natural $n$, choose in $B_{1/n}(x_0)$ a blue point $b_n$ and a red point $r_n$.
	
	If $A(b_n) <_T r_n$, then $(B_{1/n}(x_0), B_{1/n}(A(x_0)))$ contains the snake $b_n <_T A(b_N) <_T r_n$ of length $2$.
	Otherwise, $(B_{1/n}(A(x_0)), B_{1/n}(x_0))$ contains the snake $A(r_n) <_T r_n <_T A(b_n)$ of length $2$.
	
	One of these two events occurs for infinitely many different $n$, hence for arbitrarily large $n$.
\end{proof}

We will use the following properties of $\ind$ and $\dim$ (\cite{Eng}):

\begin{enumerate}
	\item (Monotonicity) if $M_1$ is a subspace of $M$ then $\ind M_1 \leqslant \ind M$;
	\item  (The sum theorem) if a separable metric space $M$ can be represented as the union of a sequence $F_1, F_2, \dots$ of closed subspaces such that $\ind F_i \leqslant n$ for $i = 1, 2,\dots$, then $\ind M \leqslant n$;
	\item (The coincidence theorem) for every separable metric space $M$ we have $\ind M = \dim M$.
\end{enumerate}

\subsection{Some basics of the topological dimension theory}

\begin{definition}
Let $X$ be a set and $\mathcal{A}$ a family of subsets of $X$.
By the \emph{order} of the family $\mathcal{A}$ we mean the largest integer $n$ such that the family $\mathcal{A}$ contains $n+1$ sets with a non-empty intersection; in no such integer exists, we say that the family $\mathcal{A}$ has order $\infty$.
A cover $\mathcal{B}$ is a \emph{\it refinement} of another cover $\mathcal{A}$ of the same space if for every $B\in \mathcal{B}$ there exists an $A\in \mathcal{A}$ such that $B \subset A$.
\end{definition}

\begin{definition}
	To every metric space $M$ one assigns the \emph{covering dimension} of $M$, denoted by $\dim M$, which is $\infty$ or an integer larger that or equal to $-1$;   the definition of the function $\dim$ is defined by the following conditions:
	\begin{enumerate}
		\item $\dim M \leqslant n$, where $n = -1,0,1,\dots,$ if every finite open cover of the space $M$ has a finite open refinement of order $\leqslant n$;
		\item $\dim M = n$ if $\dim M \leqslant n$ and $\dim M > n$, i.e. the inequality $\dim M \leqslant n-1$ does not hold;
		\item $\dim M = \infty$ if $\dim M > n$ for $n = -1, 0, 1, \dots$
	\end{enumerate}
	
\end{definition}

\begin{definition}
	To every metric space $M$ one assigns the \emph{small inductive dimension} of $M$, denoted by $\ind M$, which is $\infty$ or an integer larger than or equal to -1; the definition of the function $\ind$ is defined by the following conditions:
	
	\begin{enumerate}
		\item $\ind M = -1$ if and only if $M = \emptyset$;
		\item $\ind M \leqslant n$, where $n = 0, 1, \dots$, if for every point $x\in M$ and each $r>0$ there exists an open set $U \subset M$ such that $x\in U \subset B_{\varepsilon}(x)$ and the dimension of its boundary $\ind \partial U \leqslant n-1$;
		\item $\ind M = n$ if $\ind M > n$; 
		\item $\ind M = \infty$ if $\ind M > 0, 1, 2,\dots$
	\end{enumerate}
\end{definition}

 \section{Proof of Theorem~\ref{thm:A}}

\begin{lemma}
Let $(M,T)$ be an ordered metric space and let 
$$\Snake(M,T) \leqslant n.$$ 
Let $M$ be a disjoint union $M= A \sqcup B$, and suppose that $x_1 <_T x_2$ for all $x_1\in A, x_2\in B$.
Then there exists an order $T'$ on the boundary $\partial A$ such that $\Snake(\partial A, T') \leqslant n-1$.
\end{lemma}

\begin{proof}
For distinct points $x_0, y_0 \in \partial A$ 
we set $x <_{T'} y$ if for any $\varepsilon > 0$ there exist $x\in B_{\varepsilon}(x_0) \cap A$  and some $\delta > 0$ such that 
$x <_T B_{\delta}(y_0) \cap A$, i.e.
$$ x <_T y \text { for any } y \in B_{\delta}(y_0) \cap A.
$$

We first show that for any $x_0, y_0$ at least one of the relations $x_0 <_{T'} y_0$ and $y_0 <_{T'} x_0$ holds. 
Denote $X_k = B_{1/k}(x_0)\cap A$, $Y_k = B_{1/k}(y_0)\cap A$.
For all sufficiently large $k$, the pairs $(X_k,Y_k)$ and $(Y_k,X_k)$ do not contain snakes of length $n+1$.
Consider the largest $m$ such that there is a sequence of points
$t_1 <_T t_2 <_T \dots <_T t_m$ such that 

\noindent $(i)$ all $t_i$ with odd $i$ belong to $X_k$ and all $t_i$ with even $i$ belong to $Y_k$
\begin{center}{or}
\end{center}

$(ii)$ all $t_i $ with odd $i$ belong to $Y_k$ and all $t_i$ with even $i$ belong to $X_k$.

holds.

\noindent If $t_1 \in X_k$, then $t_1 <_T y$ for any $y\in Y_k$, otherwise the sequence $y <_T t_1 <_T \dots <_T t_m$ is longer.
Similarly, if $t_1 \in Y_k$, then $t_1 <_T x$ for 
any $x\in Y_k$. If the situation $(i)$ holds for infinitely many $k$'s, 
then $x_0 <_{T'} y_0$, otherwise $y_0 <_{T'} x_0$.

Assume that for $m > 1$ there exists a cycle
$$x_0 <_{T'} x_1 <_{T'} \dots <_{T'} x_m = x_1
$$
Take $\varepsilon_0 > 0$.
The set $B_{\varepsilon}(x_0)\cap A$ contains a point $y_0$ such that $y_0 <_T B_{\varepsilon_2}(x_1)\cap A$ for some $\varepsilon_2 < \varepsilon_1$. Similarly, in $B_{\varepsilon_1}(x_1)\cap A$ there is a point $y_1$ such that $y_1 <_T B_{\varepsilon_2}(x_2)\cap A$ for some $\varepsilon_2 < \varepsilon_1$, and so on, we construct a sequence 
$$y_0 <_T y_1 <_T y_2 <_T \dots$$

The points $y_0, y_m, y_{2m},\dots$ belong to $B_{\varepsilon_0}(x_0)\cap A$, the points $y_1, y_{m+1}, y_{2m+1}, \dots$ belong to $B_{\varepsilon_1}(x_1)\cap A$.
Since $\varepsilon_0$ was arbitrary, $\Snake(x_0,x_1) = \infty$.
This contradiction shows that the ralations $x<_{T'}y$ and $y<_{T'}x$ cannot hold simultaneously and that the relation $<_{T'}$ is transitive.
So $T'$ is indeed an order relation.

Finally, assume that $\Snake_{T'}(x,y) \geqslant n$ for some points $x,y \in \partial A$.
Take an arbitrary $\varepsilon>0$.
Then there is a sequence of points
$$
t_0 <_{T'} t_1 <_{T'} \dots <_{T'} t_n,
$$

the points with odd indices belong to $B_{\varepsilon}(x) \cap \partial A$, and the points with even indices belong to
$B_{\varepsilon}(y) \cap \partial A$.

Choose $\delta > 0$ and as before, we can find a sequence of points $r_0 <_T r_1 <_T \dots <_T r_n$ such that $y_i \in B_{\delta}(t_i)\cap A$ for $i = 0, 1, \dots,n$.
Both $B_{\varepsilon}(x)$ and $B_{\varepsilon}(y)$ contain points from $B$, hence $(B_{\varepsilon + \delta}(x), B_{\varepsilon + \delta}(y))$ contains a snake of length $n+1$.
This leads us to a contradiction.
\end{proof}

\begin{definition}
	Let $(M,T)$ be an ordered metric space. 
	We call a subset $X\subset M$ \emph{$T$-convex} if for any triple of points $x<_Ty<_Tz$ such that $x,z \in X$ the point $y$ also $\in X$.
	For a subset $A \subset M$ we define its \emph{$T$-convex hull} as the smallest $T$-convex subset of $M$ that contains $A$.
\end{definition}

\begin{lemma}
	Let $(M,T)$ be an ordered compact metric space such that $\Snake(M,T) < \infty$. Let $x_0 \in M$ and $r>0$.
	Then there is a finite family of $T$-convex subsets $X_1,\dots, X_m$ such that their union covers $M - B_{r}(x_0)$ and does not intersect some neighbourhood  $B_{\varepsilon}(x_0)$ of $x_0$.
\end{lemma}

\begin{proof}
Let $\Snake(M,T) = s-1 < \infty$.
For any $y \neq x_0$ we can choose $\varepsilon_y > 0$ such that 
$(B_{\varepsilon_y}(y), (B_{\varepsilon_y}(x_0))$ does not contain a snake of length $s$. 
The union of sets $B_{\varepsilon_y}(y)$ covers $M  - \{x_0\}$.
The set $N = \{x\in X: d(x_0, x) \geq r\}$ is compact, so
there exists $\varepsilon > 0$ and a finite set $Y \subset M$ such that
\begin{enumerate}
	\item $N \subset \bigcup_{y\in Y} B_{\varepsilon}(y)$,
	\item for any $y\in Y$ the pair $(B_{\varepsilon}(y), B_{\varepsilon} (x_0))$ does not contain a snake of length $s$.
\end{enumerate}

Denote $U:=B_{\varepsilon}(x_0)$.
We introduce an equivalence relation on the points of the set $N$:
$$a\sim b \text{ if the }T\text{-convex hull of the set } \{a,b\} \text{ does not intersect }U.$$

It is easy to see that $\sim$ is indeed an equivalence relation.
Suppose that for some $y\in Y$ the ball $B_{\varepsilon}(y)$ intersects at least $s+1$ of $\sim$-equivalence classes.
Choose a point in each class: 
$$c_0 <_T c_2 <_T \dots <_T c_s.$$ 
Since $c_0$ and $c_1$ belong to different $\sim$-equivalence classes, $U$ contains a point $b_0$ such that $c_0 <_T < b_0 <_T c_1$.
Similarly, we found points $b_1, b_2, \dots, b_{s-1}$ and get that
$(B_{\varepsilon}(y), B_{\varepsilon}(x_0))$ contains a snake of length $2s$, we obtain a contradiction.

Since each $B_{\varepsilon}(y)$ intersects at most $s$ $\sim$-equivalence classes, 
the total number of classes $m$ is finite and $m \leqslant s|Y|$.
Hence the $T$-convex hulls of these $\sim$-equivalence classes satisfy the requirements of the Lemma.
\end{proof}

Now we are ready to prove Theorem \ref{thm:A}.

\begin{proof}
We proceed by induction on $n$.
If $\Snake(M,T) \leqslant 0$, then $|M| \leqslant 1$ and $\dim M \leqslant 0$.
Suppose the statement is true for $n-1$.
We use the definition of small inductive dimension, we want to show
that for every point $x\in M$ and each $r>0$ there exists an open set $U \subset M$ such that $x\in U \subset B_{r}(x)$ and $\ind \partial U \leqslant n-1$.

Using Lemma 2 we find $T$-convex sets $X_1,\dots, X_m$ such that their union covers $M - B_r(x)$ and does not intersect some neighborhood $B_{\varepsilon}(x) $.
Let $X$ be the closure of the union of all $X_i$, $i = 1,\dots, m$.
We succeed if we show that $\ind \partial X \leqslant n-1$.

Since $\partial X \subset \cup \partial X_i$, it's enough to show
that $\ind B \leqslant n-1$ for any $T$-convex set $B$ and to use the monotonicity property of $\ind$.

But indeed, $M$ can be partitioned into three sets $A$, $B$ and $C$ such that all points of $B$ are $T$-less than all points of $A$ and all points of $C$ are $T$-greater than all points of $A$.
Since $\partial(A \cup B)$ and $\partial(A \cup C)$ are compact spaces, by Lemma 1 and the induction hypothesis we have $\ind(\partial(A \cup B)) \leqslant n-1 $ and $\ind(\partial(A \cup C)) \leqslant n-1$.
So $\ind \partial B \leqslant n-1$.
\end{proof}

\section{Proof of Theorem \ref{thm:B}}

Let $M$ be a compact metric space with $\dim M \leqslant n$.

First we construct a sequence $\mathcal{U}_i$, $i = 0, 1, \dots$ of open covers of the space $M$ satisfying the following properties

\begin{enumerate}
	\item $\mathcal{U}_i$ is a refinement of $\mathcal{U}_{i-1}$ for $i = 1, 2, \dots$;
	\item every $U \in \mathcal{U}_i$ has diameter $\leqslant 2^{-i}$;
	\item for any $i$ and any point $x\in M$, some neighbourhood $B_{\varepsilon}(x)$ of $x$  intersects no more than $n+1$ sets from $\mathcal{U}_i$.
\end{enumerate}

We construct this sequence of covers by induction.
Suppose we have already constructed $\mathcal{U}_{i-1}$, then it is easy to find a cover of $\mathcal{V}$ that satisfies properties 1 and 2 and has order $\leqslant n$.
To obtain a cover that also satisfies property 3, note that there is $r>0$ such that for any point $x\in M$ the ball $B_r(x)$ is entirely contained in some set from $\mathcal{V}$. The cover $\mathcal{U}_i$ is obtained from $i$ by replacing each $U$ by $U_{-r/2}$, obtained by subtracting from $U$ the closure of the $(r/2)$-neighborhood of its complement.

For each set $U \in \mathcal{U}_{i+1}$ choose a set $f(U)\in \mathcal{U}_i$ such that $U \subset f(U)$.
A {\it chain} is a sequence of sets $(U_0, U_1, \dots)$ such that $U_{i-1} = f(U_i)$ for $i = 1, 2, \dots$.

Let us show that each point $x\in M$ is the intersection of all elements of some chain.
Indeed, finite chains whose intersection contains $x$ form a tree with vertices of finite degrees, this thee admits arbitrarily long paths, hence from K{\H o}nig's lemma it follows that this tree contains a ray.
Fix for every point $x$ such a chain $c(x) = (U_0(x), U_1(x), \dots )$. Obviously, $c(x)\neq c(y)$ for distinct $x$ and $y$.

For each $i$ we choose an arbitrary total order $T_i$ on $\mathcal{U}_i$. 
We define a total order $T$ on $M$ by the following rule.

$x<_Ty$ if $k$ is such that

$$
 U_i(x) = U_i(y) \text{ for } i = 0, 1,\dots, k-1 \text { and }U_k(x) <_ {T_k} U_k(y),
$$
i.e. we compare corresponding chains lexicographically.

Now let  $a,b \in M$. Choose $k$ such that $2^{-k+1} < d(a,b)$.
Let $r$ be such that $B_r(a)$ and $B_r(b)$ intersect at most $n+1$ sets in $\mathcal{U}_k$.
Denote these sets $A_1,\dots A_{m_1}$ and $B_1,\dots, B_{m_2}$, $m_1, m_2 \leqslant n+1$.
None that no $A_i$ intersects with $B_j$.

Let's prove that $(B_r(a), B_r(b))$ does not contain a snake of length $2n+2$.
Assume the contrary.
The snake consists of $n+2$ points $a_1,\dots,a_{n+2}$ from $B_r(a)$
 and of $n+1$ points $b_1,\dots, b_{n+1}$ from $B_r(a)$.

For $i = 1,\dots, n+2$ the set $U_k(a_i)$ is one of $A_1,\dots,A_{m_1}$.
Thus there are $i < j$ such that
$U_k(a_i) = U_k(a_j)$.
Then $U_{k'}(a_i) = U_{k'}(a_j)$ for all $k' = 0, 1, \dots, k$.
Since $a_i <_T b_{i+1} <_T a_j$, then $U_k(b_{i+1}) = U_k(a_i)$. 
This contradiction shows that $\Snake_T(x,y) \leqslant 2n+1$ and $\Snake(M,T) \leqslant 2n+1$.

\section{Concluding remarks}

\subsection{} 
The compactification theorem says \cite{Eng} that for every separable metric space $X$ there exists a compact metric space $X'$ which contains a dense subspace homeomorphic to $X$ and satisfies $\dim X' \leqslant \dim X$.
Hence Theorem \ref{thm:B} can be extended for all separable metric spaces of covering dimension $\leqslant n$.

\subsection{} 
\begin{definition}
	A topologocal space $M$ is called \emph{totally disconnected} if for every pair $x, y$ of distinct points of $M$ there exists an open-and-closed set $U\subset M$ such that $x\in U$ and $y\in M - U$.
\end{definition}

It is easy to see that every compact  totally disconnected metric space has $\dim = 0$.

\begin{proposition}
	Let $M$ be a totally disconnected separable metric space.
	Then it admits a total order $T$ for which $\Snake(M,T) \leqslant 1$.
\end{proposition}

\begin{proof}
	Consider a sequence of open-and-closed sets $U_0, U_1,\dots$ such that for any two distinct points $x,y \in M$ at least one of the sets $U_i$ contains exactly one of $x$ and $ y$.
	For a point $x\in M$ consider a binary sequence $c(x)$, whose $i$-th term is 1 if $x\in U_i$ and 0 otherwise.
	We say that $x <_T y$ is $c(x)$ is lexicographically smaller than $y$.
	
	Suppose $c(x)$ and $c(y)$ differ at $i$-th term. There exists $r>0$ such that all points from $B_r(x)$ have the same prefix of $c(\cdot)$ of length $i$; the same for $B_r(y)$.
	In this situation $(B_r(x), B_r(y))$ does not contain snakes longer than 1.
\end{proof}

Mazurkiewicz showed \cite{Maz} the existence of completely metrizable totally disconnected separable space of arbitrary dimension $n \geqslant 1$ (and consequently of infinite dimension).
Therefore the compactness property cannot be omitted in Theorems \ref{thm:main} and \ref{thm:A}.

\subsection{}
Examples \ref{ex:1} and \ref{ex:2} show that two spaces of the same covering dimension can have different possible minimal snake numbers.
The following question arises.

\begin{question}For given $n$, what is the minimal $s(n)$ such that for any compact metric space $M$ s.t. $\dim M \leqslant n$ there exists an order $T$ for which $\Snake(M,T) \leqslant s( n)$?
\end{question}

Since in \cite{EM2} it was shown that for any order on the product of $n$ tripods there is a pair of points with snake number at least $2n$, the answer is $2n$ or $2n+1$.
It was also shown in \cite{EM1} that for $n=1$ the answer is $3$.

\end{document}